\documentclass[12pt]{amsart}
\usepackage{graphicx}
\usepackage{subfigure}
\usepackage{latexsym}
\usepackage{amsmath}
\usepackage{amssymb}

\usepackage{amscd}
\usepackage{color}
\usepackage{tikz}
\usetikzlibrary{arrows,matrix,positioning}
\newtheorem{theorem}{Theorem}[section]
\newtheorem{Def}[theorem]{Definition}
\newtheorem{Prop}[theorem]{Proposition}
\newtheorem{Lem}[theorem]{Lemma}
\newtheorem{Cor}[theorem]{Corollary}
\newtheorem{Rem}[theorem]{Remark}

\newtheorem{Exa}{Example}[section]

\newcommand{\ba}{\begin{array}}
\newcommand{\ea}{\end{array}}
\newcommand{\D}{{\mathcal D}}
\newcommand{\E}{{\mathcal E}}
\newcommand{\F}{{\mathcal F}}
\def\ba{{\bf a}}

\def\bi{{\bf i}}
\def\bj{{\bf j}}
\def\bu{{\bf u}}
\def\bw{{\bf w}}
\def\pt{\partial}
\def\dist{\textup{dist}}

\def\bv{{\bf v}}

\newcommand{\R}{{\mathbb R}}

 \numberwithin{equation}{section}
 
\begin{document}
\title{Lipschitz equivalence of self-similar sets and hyperbolic boundaries II}

\author{Guo-Tai Deng} \address{Department of Mathematics Central China Normal University Wuhan, 430079, P.R. China}\email{hilltower@163.com}
\author{Ka-Sing Lau*} \address{Department of Mathematics, The Chinese University of Hong Kong, Hong Kong, P.R. China}\email{kslau@math.cuhk.edu.hk}
\author{Jun Jason Luo**} \address{Department of Mathematics, Shantou University, Shantou, Guangdong 515063, P.R. China}\email{luojun@stu.edu.cn}

\thanks{*supported by the HKRGC grant, and  the NNSF of China (11171100, 11371382)\\
\indent **supported by NNSF of China (11301322), Specialized Research Fund for the Doctoral Program of Higher Education of China (20134402120007), Foundation for Distinguished Young Talents in Higher Education of Guangdong (2013LYM\_0028) and STU Scientific Research Foundation for Talents (NTF12016).}
\subjclass[2010]{Primary 28A80; Secondary 05C63}
\keywords{self-similar set, augmented tree, hyperbolic boundary, near-isometry, Lipschitz equivalence, quasi-rearrangeable matrix}

\date{\today}

\begin{abstract}
In \cite{LuLa13}, two of the authors gave a study of Lipschitz equivalence of self-similar sets through the augmented trees, a class of hyperbolic graphs introduced by Kaimanovich \cite{Ka03} and developed by Lau and Wang \cite{LaWa09}. In this paper, we continue such investigation. We remove a major assumption in the main theorem in \cite{LuLa13} by using a new notion of quasi-rearrangeable matrix,  and show that the hyperbolic boundary of any simple augmented tree is Lipschitz equivalent to a Cantor-type set.  We then apply this result to consider the Lipschitz equivalence of certain totally disconnected self-similar sets as well as their unions.
\end{abstract}

\maketitle

\section{\bf Introduction}

The class of hyperbolic graphs plays an important role in geometric group theory (\cite{Gr87, Wo00}).  Such graphs together with their limits ({\it hyperbolic boundaries}) have striking resemblance to the classical hyperbolic spaces.  In \cite {Ka03}, Kaimanovich  first introduced this hyperbolicity into the study of  self-similar set $K$. He initiated the notion of {\it augmented tree} by adding  more edges to the symbolic space of $K$ according to the neighboring cells.  This gives a far richer structure on the symbolic space. The idea was pursued by Lau and Wang \cite {LaWa09} (also Wang \cite{Wa12}),  they showed that for $K$ satisfying the open set condition, the augmented tree is  hyperbolic, and  $K$  can be  identified with the hyperbolic boundary of the augmented tree.  There is a large literature on random walks on hyperbolic graphs and their boundary behaviors (see \cite{Wo00} and references therein); such consideration on augmented trees can be found in \cite{Ka03, LaWa11}. In another attempt, the augmented trees and hyperbolic boundaries were used to study the Lipschitz equivalence of  self-similar sets \cite{LuLa13} and Moran fractals in \cite{Lu13}.

\medskip

Recall that two compact metric spaces $(X,d_X)$ and $(Y,d_Y)$ are said to be {\it Lipschitz equivalent}, and denoted by $X\simeq Y$, if there is a bi-Lipschitz map  $\sigma$ from $X$ onto $Y$, i.e., $\sigma$ is a bijection and  there is a constant $C>0$ such that
$$
C^{-1}d_X(x,y)\leq d_Y(\sigma(x),\sigma(y))\leq Cd_X(x,y), \quad  \quad \forall  \ x,y\in X.
$$
 Lipschitz classification of sets has been an important topic in geometry, topology and analysis.  In fractal geometry, the pioneer work was due to Cooper and Pignartaro \cite{CoPi88} and Falconer and Marsh \cite{FaMa92} on Cantor-type sets under the strong separation condition.  The recent interest was due to Rao, Ruan and Xi \cite{RaRuXi06} on their path breaking solution to a question of David and Semmes, so called the $\{1,3,5\}-\{1,4,5\}$ problem.  For the developments and the generalizations, the reader can refer to \cite{DeHe12, RaRuWa10, RuWaXi12, XiRu07, XiXi10,XiXi12, XiXi13, XiXi14} for more details.  In particular,  in \cite{XiXi13}, the Lipschitz classification of self-similar sets with exponentially commensurable contraction ratios is characterized in terms of the ideal classes in algebra.

\medskip

Let $K$ be a self-similar set generated by an iterated function system (IFS) of $N$ similitudes of equal contraction ratio, and let $X = \cup_{n=0}^\infty \Sigma^n, \Sigma = \{1, \dots , N\}$  be the associated symbolic space of words; we also use the notion  ``$N$- $\cdots$''  to emphasize on the cardinality $N$.  We denote the set of edges from the canonical tree structure by $\E_v$ (vertical edges); as a tree the boundary is a homogeneous Cantor set.  We add new edges by joining words $\bi, \bj$ in the same level $\Sigma^n$ if the corresponding cells $K_\bi$, $K_\bj$  intersect, and denote this set of edges by $\E_h$ (horizontal edges). Let $\E = \E_v\cup \E_h$, and call $(X, \E)$ a {\it self-similar augmented tree} \cite {Ka03}. We say that $(X, \E)$ is {\it simple} if there is only finitely many non-isomorphic classes of subgraphs defined by the horizontal components and their descendants  (see Definition \ref{def of simple tree}). In this case, $(X, \E)$ is hyperbolic \cite {LuLa13}, and the hyperbolic boundary $\partial X$ can be identified with $K$  (\cite {Ka03, LaWa09}).  We use $A$ to denote the incidence matrix, which describes the graph relation of the horizontal  components of $(X, \E)$.  The main theorem in \cite{LuLa13} is

\medskip

{\it If the incidence matrix $A$ is primitive, then $\partial(X, {\mathcal E})$ is Lipschitz equivalent to  $\partial(X, {\mathcal E}_v)$, which is a homogeneous $N$-Cantor set.

Moreover, if the self-similar set $K$ satisfies condition (H)(see Section 5), then $K$ is Lipschitz equivalent to the $N$-Cantor set with the same contraction ratio as the IFS.}

\medskip

 Note that in this consideration, we do not need to assume the open set condition, but it will come out as a consequence of the Lipschitz equivalence (Corollary 3.11 in \cite{LuLa13}). This augmented tree approach  provides a general and simple framework to study the Lipschitz equivalence of totally disconnected self-similar sets, and unifies many of the previous investigations. It covers most of the known examples, and also certain Moran fractals \cite {Lu13}.  In the investigation in \cite{LuLa13},  a number of questions were raised. In particular, it was asked whether the assumption that the incidence matrix $A$ is primitive can be removed, as there are simple examples that such condition is not satisfied (see Example \ref{exa5.3} in Section 5 or discussions in \cite{XiXi13}).

\medskip

In this paper we continue our investigation started in \cite{LuLa13}. Our main purpose is to remove the primitive assumption on the incidence matrix, and to extend the scope  to more general class of augmented trees, which includes the union of certain self-similar sets.  We called $(X, {\mathcal E})$ an {\it $N$-ary augmented tree} if it is a tree such that each vertex has $N$ descendants, and the horizontal edges satisfy the condition in Definition \ref {Def'}.

\medskip

\begin{theorem} \label{th1.1}
    Suppose an $N$-ary augmented tree $(X, {\mathcal E})$ is simple, then $\partial(X, {\mathcal E})$ is Lipschitz equivalent to  $\partial(X, {\mathcal E}_v)$, which is an $N$-Cantor set.
\end {theorem}

 \medskip

    By applying the theorem to  self-similar sets,  we have

 \medskip

 \begin{theorem} \label{th1.2}
    Suppose an self-similar augmented tree $(X, {\mathcal E})$ defined by an IFS ($N$ similitudes with equal contraction ration $r$) is simple and satisfies condition (H) (see Section 5),  then $K$ is Lipschitz equivalent to the $N$-Cantor set with contraction ratio $r$.
\end{theorem}

\medskip

The proof Theorem \ref{th1.1} is based on constructing  a {\it near-isometry} $\sigma$ between the augmented tree $(X, \E)$ and the tree $(X, \E_v)$ ($\sigma$ is stronger than the rough isometry in literature).  In \cite{LuLa13}, the existence of such isometry depends on the incidence matrix $A$ is primitive, which implies {\it rearrangeable}, a combinatoric property that  allows us to permute the vertices and edges of the augmented tree in order to construct $\sigma$.  Without the primitive condition as in Theorem \ref{th1.1},  we need to introduce a new notion of {\it quasi-rearrangeable}  to obtain the needed near-isometry (Sections 3, 4). In doing so, we also need to extend slightly the definition of near-isometry, together with other modifications of the augmented trees that include the unions and quotients. As a consequence, we can use Theorem \ref{th1.2} to consider some fractal sets that are not necessarily self-similar. Among those, we prove

\medskip

\begin{Prop}\label{pro1}
Let $\mathcal C$ be the standard Cantor set. Then ${\mathcal C}\cup ({\mathcal C}+\alpha)$ is Lipschitz equivalent to $\mathcal C$ if $\alpha>1$ or if $0<\alpha\le 1$ is a rational.
\end{Prop}

\medskip
There have been considerable studies on the intersections of Cantor sets  (see \cite{Fu70, DeHeWe08, ElKeMa10, FeHuRa14} and references therein). However, to our knowledge, there are few results on their unions. Proposition \ref{pro1} is perhaps a new attempt on the Cantor sets.

\medskip

The paper is organized as follows: In Section 2, we briefly review the hyperbolic graphs and the augmented trees to set up the notations,  and derive some basic properties. We define the quasi-rearrangeable matrices
in Section 3, and prove Theorem \ref{th1.1} in Section 4. Finally in Section 5, we apply the main results on the hyperbolic boundaries to self-similar sets and their unions by proving Theorem \ref{th1.2} and Proposition \ref{pro1}.

\bigskip

\section{\bf The augmented tree}

We use the same notations as in \cite{LuLa13}. Let $X$ be an infinite connected graph. For $x, y \in X$, let  $\pi(x, y)$ denote a geodesic from $x$ to $y$, and $d(x,y)$ its length. Let $o$ be a root of the graph, and let $|x| = d(o, x)$.  According to \cite{Wo00}, for $x, y \in X$, let
\begin{equation} \label{eq2.1}
    |x\wedge y|=\frac12\big(|x|+|y|-d(x,y)\big)
\end{equation}
denote the {\it Gromov product}, and  call $X$  {\it hyperbolic} (with respect to $o$) if there is $\delta \geq 0$ such that
$$
 |x\wedge y|\ge \min\{|x\wedge z|,|z\wedge y|\}-\delta \quad \text{for any} \quad x,y,z\in X.
 $$
For $a>0$ with $\exp(\delta a) -1 < \sqrt 2 -1$, we define a {\it hyperbolic metric} on $X$ by
\begin{equation}\label{eq2.2}
    \rho_a(x,y)=\delta_{x,y}\exp(-a|x\wedge y|),
\end{equation}
where $\delta_{x,y}=0,1$ according to $x=y$ or $x\ne y$. Let $\overline{X}$ be the completion of $X$ in the metric $\rho_a$. We call $\pt X= \overline{X}\setminus X$ the {\it hyperbolic boundary} of $X$.  It is clear that $\rho_a$ can be extended to $\pt X$,  and  $\pt X$ is a compact set under $\rho_a$. It is useful to identify $\xi\in\pt X$ with a {\it geodesic ray} in $X$ that converges to $\xi$.

\medskip

Let $X$ be a tree with root $o$.  It is well-known that $X$ is hyperbolic (with $\delta =0$), and the hyperbolic boundary is totally disconnected. We use $\E_v$ to denote the set of edges of $X$ ($v$ for vertical), and $X_n = \{x \in X: |x| =n\}$. We introduce some additional edges on each level of $X$.

\bigskip

\begin{Def} \label{Def'} (\cite{Ka03, LaWa09}) 
Let $(X, \E_v)$ be a tree. We call $(X, \E)$ an augmented tree if $\E=\E_v\cup \E_h$,  where $\E_h\subset (X\times X)\setminus \{(x,x):\ x\in X\}$ is symmetric and satisfies
\begin{equation} \label{Def}
(x,y)\in \E_h\Rightarrow |x|=|y|, \mbox{ and either } x^{-}=y^{-} \mbox{ or } (x^{-},y^{-})\in \E_h.
\end{equation}
($x^{-}$ is the {\it predecessor} of $x$.) We call elements in $\E_h$ horizontal edges.

Furthermore, if each vertex of $X$ has $N$ offsprings, we call  $(X, \E)$ an $N$-ary augmented tree.
\end{Def}

For an $N$-ary tree, it is obvious that we can identify  $X_n$ with $\Sigma^n$ where $\Sigma =\{1, \dots, N\}$,  and  hence $X = \cup_{n=0}^\infty X_n = \cup_{n=0}^\infty \Sigma^n$. We will use both notations whenever convenient.  For $x, y \in X$, the geodesic path of $x, y$ is not unique in general, but there is a canonical one of the form
\begin{equation} \label{eq2.3}
    \pi (x, y) = \pi(x, u) \cup \pi (u,v) \cup \pi (v,y)
\end{equation}
where $\pi (x, u), \pi (v, y)$ are vertical paths, $\pi (u, v)$ is a horizontal path,  and for any geodesic $\pi' (x,y)$, $d(o, \pi (u,v)) \leq d(o, \pi' (x,y))$.  (It can happen that there are only two parts, with $v=y$ or $x = u$.) The following is known (\cite{Ka03, LaWa09}):

\vspace {0.15cm}

{\it An augmented tree is hyperbolic if and only if there is $k>0$ such that the length of the horizontal parts of the canonical geodesics in $X$ is bounded by $k$}.

\vspace {0.15cm}

For $T\subset X_n$, the set of descendants of $T$ (including $T$ itself) is denoted by $T_\D$, i.e.,
$$
T_\D=\{x\in X:\  x|_n\in T\}
$$  where $x|_n$ is the initial segment of $x$ with length $n$.
Note that  if $T$ is connected, then $T_\D$ is a subgraph of $X$. Moreover, if $(X,\E)$ is hyperbolic, then $T_\D$ is also hyperbolic.
We say  that $T$ is an {\it $X_n$-horizontal component} if \ $T\subset X_n$ is a maximal connected subset with respect to ${\E}_h$, and denote $T$ by $\lfloor x\rfloor$ for $x \in T$.  We let $\F_n$ denote the family of all  $X_n$-horizontal components, and let $\F=\cup_{n\ge 0}\F_n$.   Note that for distinct $T,T' \in\F_n$, the subgraphs $T_{\D}, T'_{\D}$ are disjoint. We can define a graph structure on $\F$ as: $\lfloor x\rfloor$ and $\lfloor y\rfloor$ is connected by an edge if and only if  $(x, y) \in \E_v$; we denote this graph by $X_Q$ (see Figure \ref{fig2.1}).  It is clear that  $X_Q$ defined above is a tree, and we call it the {\it quotient tree} of $X$.

\bigskip

\begin{figure}[h]
    \centering
    \includegraphics[width=5in]{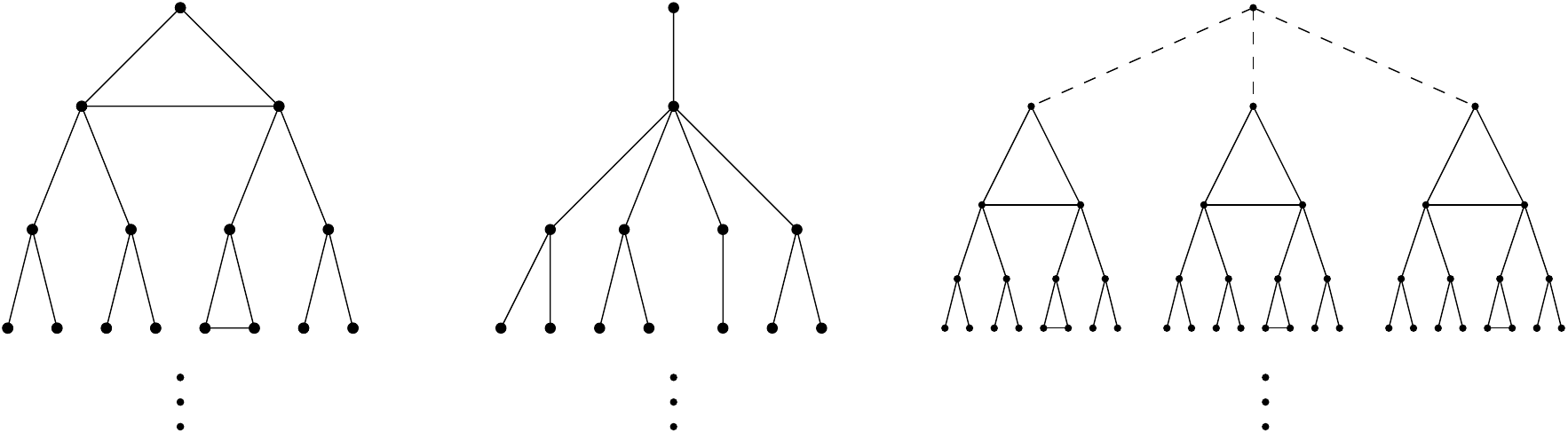}
    \caption{The augmented tree $X$, the quotient tree $X_Q$ and the union of three copies of $X$.}\label{fig2.1}
\end{figure}

For $T, T'\in\F$, we  say  that $T$ and $T'$ are equivalent, denote by $T \sim T'$ if there exists a graph isomorphism $g:\, T_\D\to T'_\D$, i.e., the map $g$ and the inverse map $g^{-1}$ preserve the vertical and horizontal edges of $T_\D$ and $T'_\D$. We denote the equivalence class by $[T]$.

\bigskip

\begin{Def}\label{def of simple tree}
We call an augmented tree $(X, \E)$  {\it simple} if the equivalence classes in $\F$ is finite. Let $[T_1],\dots,[T_m]$ be the equivalence classes in $X \setminus \{o\}$, and let $a_{ij}$  denote the cardinality of the horizontal components of offsprings of $T \in [T_i]$ that belong to $[T_j]$.  We call $A =[a_{ij}]$ the incidence matrix of $(X,\E)$.
\end{Def}

\bigskip

The above definition is a modification from Definition 3.3 of \cite{LuLa13}  (there is an oversight there, nevertheless this adjustment does not affect the proofs in \cite{LuLa13}).  We also adjust slightly the incidence matrix from the previous one ((3.3) in \cite{LuLa13}), as in here $[o]$ is not counted in $A$ as the initial one (it is still possible that there is $[T_j] = [o]$ for some $j$). This change of $A$ does not make any difference for the boundary, but will be more convenient when we consider the subgraph $T_\D$. It can be verified easily that $a_{ij}$ is independent of the choice of $T$.

\bigskip

Note that the incidence matrix $A$ and the quotient tree $X_Q$ are related as the following:  for each  $T = \lfloor x\rfloor \in X_Q$,  say $T\in [T_i]$ for some $i$, then $\lfloor x\rfloor$ has a total of $\sum_j a_{ij}$ offsprings in $X_Q$; for each $j$, there are exactly $a_{ij}$ (ignore  those $=0$) of them that are roots of isomorphic subtrees of $X_Q$. In fact, $X_Q$ is the induced tree by the graph directed system defined by $A$ \cite{MaWi88}.

\begin{Prop} \label{th2.3}
Every simple augmented tree $(X, \E)$ is hyperbolic. Moreover, $ \pt X  \simeq \pt X_Q $, and both of them are totally disconnected.
\end{Prop}

\begin{proof}
That a simple augmented tree is  hyperbolic was proved in Proposition 3.4 in \cite{LuLa13}. Basically, it follows from  the fact that the length of  horizontal components is uniformly bounded, hence the horizontal part of a geodesic is uniformly bounded, which yields the hyperbolicity of the augmented tree $X$.

 To show  that $\pt X_Q \simeq \pt X$, we note that
$(X, \E)$ is simple, there exists $k>0$ such that the number of vertices in each $T_i$ is bounded by $k$. For $x \in X$, let $\lfloor x\rfloor$ denote the horizontal component that contains $x$.  We define a projection $\tau : X \to X_Q$ by $\tau (x) = \lfloor x\rfloor$. Note that for any $x, y \in X$, the canonical geodesic, as in (\ref{eq2.3}), is $\pi (x, y) = \pi (x, u) \cup \pi (u,v) \cup \pi (v,y).$  This implies $\lfloor u\rfloor = \lfloor v \rfloor$, and $\pi (\lfloor x\rfloor, \lfloor y \rfloor) = \pi (\lfloor x\rfloor,\lfloor u \rfloor) \cup \pi (\lfloor u \rfloor, \lfloor y\rfloor)$. Since
$d(u, v) \leq k$, by (\ref{eq2.1}) and (\ref{eq2.2}), we have
\begin{equation} \label{eq2.4}
\big |\lfloor x\rfloor \wedge \lfloor y\rfloor \big | \leq |x \wedge y| \leq \big |\lfloor x \rfloor \wedge \lfloor y\rfloor \big | + k,
\end{equation}
and for $\lfloor x\rfloor \ne \lfloor y\rfloor$,
\begin{equation} \label{eq2.5}
c\rho_a (\lfloor x\rfloor , \lfloor y\rfloor) \leq \rho_a(x, y) \leq  \rho_a(
\lfloor x\rfloor, \lfloor y\rfloor)
\end{equation}
where $c= e^{-ka}$. Hence we can extend $\tau : \overline X \to \overline X_Q$ continuously. It is clear that $\tau: \pt X \to \pt X_Q$ is surjective.  We claim that it is also one-to-one.  Note that in a hyperbolic boundary, two geodesic rays $\pi (x_1, x_2, \dots)$ and $\pi (y_1, y_2, \dots )$ represent the same $\xi \in \pt X$ if and only if $|x_n\wedge y_n| \to \infty$  as $n\to\infty$ \cite{Wo00}.  Hence  for $\xi \ne \eta$ in $\pt X$, there exist geodesic rays $\pi (x_1, x_2, \dots)$ and $\pi (y_1, y_2, \dots )$ representing $\xi$ and $\eta$ respectively and $|x_n \wedge y_n| \not \to \infty$ as $n\to\infty$.  It follows from (\ref{eq2.4}) that $\big|\lfloor x_n\rfloor \wedge \lfloor y_n\rfloor\big | \not \to \infty$. This implies $\lfloor \xi \rfloor\ne \lfloor \eta \rfloor$ in $\pt X_Q$. The conclusion that $\pt X \simeq \pt X_Q$ follows by extending (\ref{eq2.5}) to the boundaries.

Since $X_Q$ is a tree, whose boundary is a Cantor-type set, it follows that both of $\pt X$ and $\pt X_Q$ are totally disconnected by the above argument. 
\end{proof}

\bigskip

\begin{Cor}\label{th2.4}
Let $X,Y$ be two simple augmented trees and have the same incidence matrix $A$. Then $\pt X\simeq\pt Y$.
\end{Cor}

\begin{proof}  
It follows from the assumption that $X_Q$ and $Y_Q$ are graph isomorphic so that $\partial X_Q\simeq \partial Y_Q$. The corollary follows from Proposition \ref{th2.3}.
\end{proof}

\bigskip

\begin{Def} \label{def2.1}
Let $X,Y$ be two hyperbolic graphs. We say that $\sigma$ is a  near-isometry of $X$ and $Y$ if there exist finite subsets $E\subset X$, $F\subset Y$, and $c>0$  such that $\sigma: X \setminus E \to Y\setminus F$ is a bijection and satisfies
$$
\big| d(\sigma(x),\sigma(y))-d(x,y)\big|<c.
$$
\end{Def}

\bigskip

We remark that this definition of near-isometry is a slight relaxation of the one in \cite{LuLa13} by allowing an exception of finite sets. Actually, we can allow the sets $E, F$ to be countable as long as in the boundaries, the limit points from $E$ and $F$ are the limit points of $X\setminus E$ and $Y\setminus F$ respectively.  The proof of the following proposition is the same as in \cite{LuLa13} with some obvious modifications.

\bigskip

\begin{Prop}\label{th2.6}
Let $X,Y$ be two hyperbolic augmented trees.  Suppose there exists a near-isometry from $X$ to $Y$, then $\pt X\simeq \pt Y$.
\end{Prop}

\bigskip
 The following is a crucial algebraic property of a simple $N$-ary augmented tree, the proof follows easily from the definition.

\medskip

\begin{Prop}\label{th2.7}
Let  $(X, \E)$ be a simple  $N$-ary augmented tree, let $\{[T_1],\dots, [T_m]$ be the equivalence classes with incidence matrix $A$, and let $\bu =[u_1, \dots, u_m]^t$ where $u_i=\#T$ for $T\in [T_i]$. Then $A\bu = N \bu$.
\end{Prop}

\bigskip

Let $X_i, 1\le i \le \ell$ be augmented trees  with roots $o_i$.  Let $\widehat{X}=(\cup_{i=1}^\ell X_i)\cup \{o\}$ where $o$ is an additional vertex. We equip $\widehat{X}$  with an edge set $\widehat{\mathcal E}$ that includes all ${\mathcal E}_i$ and the new edges joining $o$ and $o_i$. Then $(\widehat{X}, \widehat{\mathcal E})$ forms a new connected graphs and each $(X_i, {\mathcal E}_i)$ becomes its subgraph (see Figure \ref{fig2.1}). We call $(\widehat{X}, \widehat{\mathcal E})$ {\it the union of $\{X_i\}_{i=1}^\ell$}. Occasionally we use $\cup_{i=1}^\ell(X_i, {\mathcal E}_i)$ or $\cup_{i=1}^\ell  X_i$ to denote $(\widehat{X}, \widehat{\mathcal E})$ for clarity.  The following proposition is useful.

\bigskip

\begin{Prop}\label{th2.8}
Let  $(X, \E)$ be an  $N$-ary augmented tree such that $\partial (X, \E)\simeq \partial (X, \E_v)$. Suppose $(X_i,\E_i), 1\le i \le \ell $, are copies of $(X,{\E})$, and $(\widehat X,\widehat {\E})$ is the union of $\{(X_i,\E_i)\}_{i=1}^\ell$. Then $\pt(\widehat X,\widehat\E)\simeq \pt(X,\E)$.
\end{Prop}

\begin{proof}
It is easy to see that for $\partial (X_i, \E)\simeq \partial (X_i, \E_v)$, the disjoint union implies $\pt(\widehat X,\widehat\E)\simeq \pt (\widehat X,\widehat\E_v)$. Hence it suffices to prove the proposition for $\pt(\widehat X,\widehat\E_v)\simeq \pt(X,\E_v)$.

\vspace{0.1cm}

Let $X=\bigcup_{n=0}^{\infty}\Sigma^n$ where $\Sigma=\{1,\dots, N\}$. Consider a subset of vertices of $X$:
$$
I=\{{\mathbf i_s}\}_{s=1}^{\infty}=\{1,\dots, N-1; N1,\dots,N(N-1);N^21,\dots,N^2(N-1);\dots\}.
$$
Similarly, for the vertices of the union $\widehat{X}$, denote by
$$
J:=\{{\mathbf j_s}\}_{s=1}^{\infty}:=\{o_1,\dots, o_{\ell-1}\}\cup \{o_{\ell}{\mathbf i}_s: {\mathbf i}_s\in I\}.
$$
Define a map $\sigma: I\to J$ by $\sigma({\mathbf i}_s)={\mathbf j}_s$  (see Figure \ref{union tree}),  and extend it to
$$
\sigma: {X}\setminus \{o\} \to \widehat X\setminus \{o\}
$$
by $\sigma({\mathbf i}_s{\mathbf u})={\mathbf j}_s{\mathbf u}$ for ${\mathbf u}\in X$, and $\sigma(N^{i+1})= o_\ell N^i$ for $i=0,1,2,\dots$ (this last part of $\sigma$ is not essential in view of the remark after Definition {\ref {def2.1}}). Then the map is bijective and  satisfies
$$
\big|d(\sigma(x), \sigma(y))-d(x,y) \big|\leq \left[{\ell}/{N}\right]+1, \quad \forall  \  x,y\in {X}\setminus\{o\},
$$
where $[{\ell}/{N}]$ denotes the largest integer not greater than ${\ell}/{N}$. This can be verified immediately on $I$ first, and then for arbitrary $x,y$. Therefore, $\sigma$ is a near-isometry, and  the result follows by Proposition \ref{th2.6}.

\begin{figure}[h]
    \centering
    \includegraphics[width=5in]{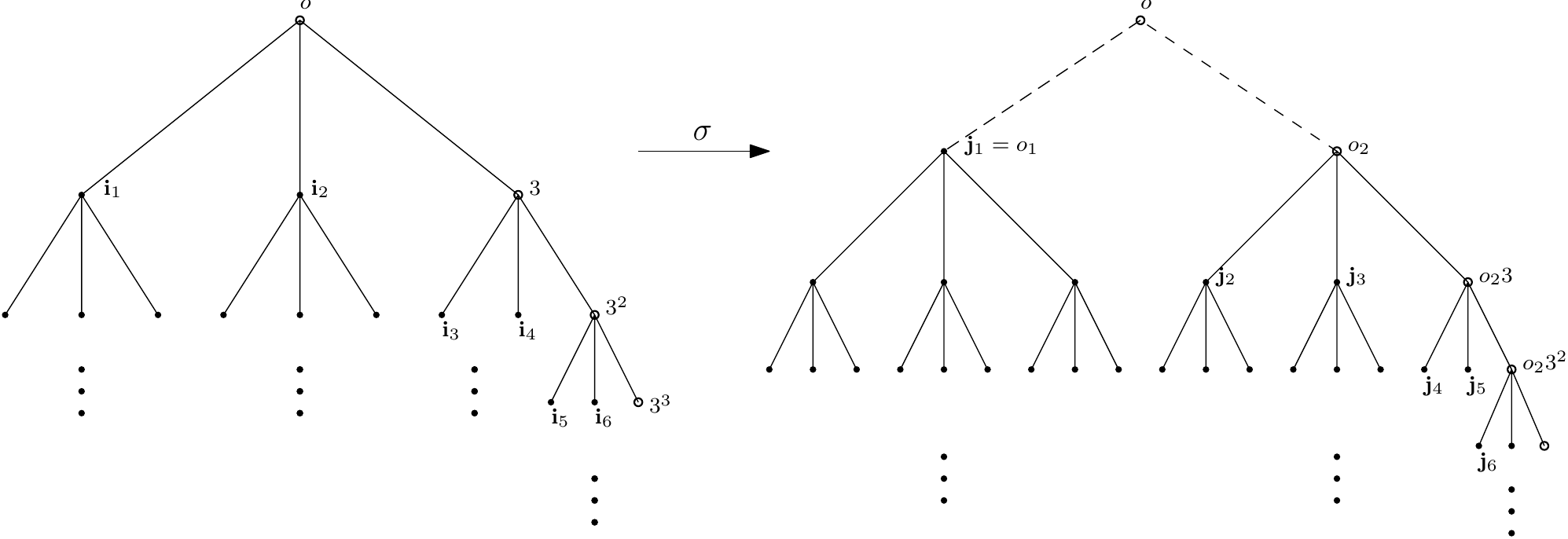}
\caption{An illustration of the map $\sigma: I\to J$ by letting $N=3, \ell=2$.}\label{union tree}
\end{figure}

\end{proof}

\bigskip

There is another useful variance of an augmented tree. Let $(X, \E)$ be an $N$-ary augmented tree. For $k>1$, we write $X^{(k)}=\cup_{n=0}^\infty X_{kn}$, then $X^{(k)}$ is a $kN$-ary tree. We define the horizontal edges on the $n$-th level of $X^{(k)}$ to be the same as the the $kn$-th level in $X$, and denote the  induced edge set by $\E_h$ as well. Let $\E = \E_v \cup \E_h$ on $X^{(k)}$, then the following proposition is immediate.

\bigskip

\begin{Prop}\label{th2.9}
Let $X$ be an $N$-ary tree, and $(X, \E)$ be a simple augmented tree with an incidence matrix  $A$. Then for $(X^{(k)}, \E)$ defined as above, the incidence matrix is $A^k$, and if we take the hyperbolic metric $\rho_a$ and $\rho_{ka}$ on the respective spaces, then $\pt (X, \E) = \pt(X^{(k)}, \E)$
\end{Prop}

\bigskip

To conclude this section, we remark that it is rather flexible to choose the horizontal edges to form an augmented tree (see Definition \ref{Def'} and \cite{LaWa09, LuLa13, Wa12}).

\bigskip

\noindent {\bf Example 2.10.}
{\it Let $X=\cup_{n\ge0}\Sigma^n$, $\Sigma=\{1,2\}$. Let $\E_h=\{(11,12)\}$ and $\E=\E_v\cup \E_h$.   It is easy to see that the equivalence classes are  $[1],[2],[11,12],$  and the incidence matrix is
$$A=\begin{bmatrix}
          0&0&1\\
          0&2&0\\
          0&4&0
    \end{bmatrix}.$$
Hence $(X,\E)$ is simple and its hyperbolic boundary is Lipchitz equivalent to the one of $(X,\E_\upsilon)$,  but $(X,\E)$ can not be induced by an IFS.}

\bigskip

\noindent {\bf Example 2.11.} {\it
Let $X=\cup_{n\ge0}\Sigma^n, \Sigma=\{1,2\}$, and let $A\subset {\Bbb N}$ be a non-periodic infinite set. Let ${\mathcal E} = {\mathcal E}_v \cup \E_h$, where
$$
\E_h=\{(1^p,1^{p-1}2):\  p\in A\}.
$$
Then $(X,\E)$ is a hyperbolic augmented tree. There are infinitely many  equivalence classes
$$\{[2]\}\cup\{[1^p]:\ p \not \in  A\} \cup\{[1^p,1^{p-1}2]:\ p \in A \}.$$
Hence $(X,\E)$ is not simple.  But the identical map $\sigma: (X,\E_v)\to (X,\E)$ is  a near-isometry, which implies $\pt(X, \E_v)\simeq \pt(X, \E)$ by Proposition \ref{th2.6}.}

\bigskip

\section{\bf Quasi-rearrangeable matrices}

We see in Proposition \ref{th2.3} that the hyperbolic boundary of a simple augmented tree is totally disconnected. In order to show that its boundary is also Lipschitz equivalent to a homogenous Cantor set, a combinatoric device to rearrange vertices is needed.  This idea was introduced in \cite {RaRuXi06}, reformulated and investigated in \cite{DeHe12} and  \cite{LuLa13}: {\it Consider a set of vertices that are connected by edges, the number of connected components with size $u_i$ is $a_i$. For $N>0$, under what condition  can we rearrange (but not breaking) these components into groups such that each group has $N$ vertices?}  In this case, we can put this group of vertices as the $N$ descendants of one vertex. We will make use of  this property inductively to construct the near-isometry of the $N$-ary augmented tree with an $N$-ary tree.

\bigskip

\begin{Def}\label{th3.1}
Let $\ba=[a_1,\dots,a_m]$ and $\bu=[u_1,u_2,\dots, u_m]^t$ be in ${\mathbb N}^m$. For $N>0$, we say that $\ba$ is $(N,\bu)$-rearrangeable if there exists $p>0$ and a non-negative integral ${p\times m}$ matrix $C$ (rearranging matrix) such that
\begin{equation}\label{eq3.1}
\ba=\mathop{\underbrace{[1,\ldots,1]}\limits_{p}}C\ \mbox{ and }\ C\bu = \mathop{\underbrace{[N,\ldots, N]^t} \limits_{m}}.
\end{equation}
(In this case $\ba\bu=pN$.) We say that  $\ba$ is $(N,\bu)$-quasi-rearrangeable if the second identity is replaced by  $C\bu \leq [N,\dots,N]^t$.

A matrix $A$ is said to be $(N,\bu)$-rearrangeable (quasi-rearrangeable) if each row vector in $A$ is $(N,\bu)$-rearrangeable (quasi-rearrangeable). (Note that the $p$ and $C$ in each row may be different.)
\end{Def}

\medskip

To realize the above definition, let us assume that there are $m$ different kinds of objects, each kind has cardinality $a_i$ and each one of the same kind has weight $u_i$, hence the total weight is  $\sum_i a_iu_i = pN$.  The rearranging matrix $C$ is a way to divide these objects into $p$ groups (first identity in (\ref {eq3.1})) such that every entry of a row represents the number of each kind in the group, and the total weight of the objects in the group is $N$ (the second identity in (\ref {eq3.1}))
\begin{align*}
\begin{matrix}
C\ = \
\end{matrix}
&{\small \begin{tikzpicture}[baseline=-0.5ex]
\matrix[matrix of math nodes,left delimiter={[},right delimiter={]},inner sep=2.5pt, column 2/.style={black!50!black},
 ampersand replacement=\&] 
 {
   c_{11} \& \cdots \& c_{1j} \& \cdots \& c_{1m} \\
\vdots \&   \&  \vdots  \&  \& \vdots \\
c_{i1} \& \cdots \& c_{ij}  \& \cdots \& c_{im} \\
\vdots \&   \&   \&  \& \vdots \\
c_{p1} \& \cdots \& c_{pj} \& \cdots \& c_{pm}  \\
};
\draw (-0.05,0) ellipse (13pt and 48pt);
\draw (0,0) ellipse (55pt and 10pt);
\end{tikzpicture}}
\begin{matrix}
\quad  \mbox{$i$-th group.}
\end{matrix} \\
&
\begin{matrix}
 & &  & & & \!\!{\mbox{sum}\ a_j} &  \\
\end{matrix}
\end{align*}

The next lemma is a basic criterion given in \cite{DeHe12} to determine  a vector to be rearrangeable (see also \cite{LuLa13}).

\begin{Lem}\label{th3.2}
Let $\ba=[a_1,\dots,a_m]$, $\bu=[u_1,\dots,u_m]^t$ be in ${\Bbb N}^m$. Suppose $\ba\bu=pN$ and  all $a_i$, are sufficiently large compare to all $u_j$,  say,
\begin{equation} \label{eq3.2}
 a_i > p^2({\sum}_{j=1}^m u_j)({\prod}_{j=1}^m u_j),  \qquad 1\leq i \leq m.
\end{equation}
Then $\ba$ is $(N,\bu)$-rearrangeable if and only if  $\gcd(\bu)$ divides $N$.  In this case, the rearranging  matrix $C$ is of size $p\times m$.
\end{Lem}

Here $\gcd(\bu)$ is the greatest common divisor of $u_1,\dots, u_m$. Intuitively, if all the weights $u_j$ are small, and there are enough objects $a_i$'s to maneuver, then it is possible to round up the group to be with weight $N$. Lemma \ref{th3.2} yields the following useful sufficient condition for rearrangement, which applies to the incidence matrix  (see Proposition \ref {th4.2}, Lemma \ref{th4.3}).

\bigskip

\begin{Prop}\label{th3.3}
Let $A$ be an $m\times m$ primitive matrix (i.e., there exists $n>0$ such that $A^n>0$), and $\bu \in {\Bbb N}^m$.  Let $ u  = \gcd(\bu) $,

\ (i) if $A\bu= N\bu$, then there exists $k>0$ such that $A^k$ is $( uN^k,\bu)$-rearrangeable;

(ii)  if $A\bu\leq N\bu$, then there is an integer $k>0$ such that $A^k$ is $(uN^k,\bu)$-quasi-rearrangeable.

\vspace {0.1cm}
 In both cases, the corresponding rearranging matrix $C_i$ for each row of $A^k$ is of size $(u_i/u) \times m$.
\end{Prop}

\begin{proof}
Let $\ba_i := \ba_i^{(k)}$ denote the $i$-th row of $A^k$. As $\bu$ is the $N$-eigenvector of $A$, it follows that
$$
\ba_i \bu = u_iN^k := pN'
$$
where $p = u_i/u$ and $N' = u N^k$.  From the primitive property of $A$, we can find an integer $k>0$ such that each entry of $A^k = [a_{ij}^{(k)}]$ is sufficiently large so that  (\ref{eq3.2}) is satisfied. Hence by Lemma \ref{th3.2}, \ $\ba_i$ is $(uN^k, \bu)$-rearrangeable, and (i) follows.

\vspace {0.1cm}

To prove (ii), we assume that $A\bu\ne N\bu$. Choose $n$ large enough such that $A^n\bu<N^n\bu$,  and let $\bw:=N^n\bu-A^n\bu > \bf 0$. Suppose $\bu=[u_1,\dots,u_m]^t$. Let $\bu'=[u_1,\dots,u_m,1]^t$ and
$$A'=\begin{bmatrix}
A^n & \bw \\
{\bf 0}& N^n
\end{bmatrix}.$$
It is direct to check that $A'\bu'=N^n\bu'$. If we denote the $i$-th row of $A^{nk}$ and $A'^k$ by $\ba_i$ and $\ba_i'= [\ba_i, a_{i,m+1}]$ respectively, then $\ba_i'\bu'=u_iN^{nk}$. It follows from the above (and Lemma \ref{th3.2})  that $\ba'_i$ is $(uN^{nk},\bu')$-rearrangeable, in the sense that for the $i$-th row vector $\ba_i'$ with $i\le m$, there exists a $(u_i/u)\times(m+1)$ non-negative matrix $C'_i$ satisfying
$$
\ba_i'={\bf 1}C_i', \quad \mbox{and} \quad C_i'\bu'=[uN^{nk},uN^{nk},\dots,uN^{nk}]^t.
$$
Let $C_i$ be obtained by deleting the last column of $C_i'$, then
$$
\ba_i={\bf 1}C_i,\quad  \mbox{and}\quad C_i\bu \leq [uN^{nk},uN^{nk},\dots,uN^{nk}]^t
$$
which yields (ii).
\end{proof}

In view of Lemma \ref {th3.2} and the proof of the above proposition, we also have

\medskip

\begin{Cor} \label {th3.4} Under the assumption in Proposition \ref{th3.3}, if further gcd$(\bu)$ divides $N$, then we can conclude that $A^k$ is $(N^k,\bu)$-rearrangeable in (i),  and $( N^k,\bu)$-quasi-rearrangeable in (ii).
\end{Cor}

\bigskip

\section{\bf Proofs of the main results}

Let $A=[a_{ij}]\in M_m({\mathbb Z})$ be a non-negative matrix, and $A^n=[a_{ij}^{(n)}]$. We say that $A$  is {\it primitive} if $A^n>0$ for some $n>0$, and is {\it irreducible} if for any entry $a_{ij}$, there exists $n>0$ such that $a_{ij}^{(n)}>0$. In matrix theory, it is well-known that for any non-negative matrix $A$, it can be brought into the form of the upper triangular block by a permutation matrix $P$,
$$
P^tAP=\left[\begin{array}{cccccc}
A_1 &  & {\bf *}  \\
 &  \ddots & \\
  {\mathbf 0} &  &  A_r
\end{array}\right]
$$
 where each $A_i$ is a square matrix that is either irreducible or zero, $i=1,\dots, r$. The following is a stronger result that for certain power $A^\ell$, the block matrices are primitive, if not zero.

 \bigskip

\begin{Lem}\label{th4.1}
Let $A$ be a non-negative matrix, then we have

\ (i) if $A^n$ is irreducible for any $n\geq 1$, then $A$ is primitive;

(ii) there is $\ell\geq 1$ such that the block matrices lying in the diagonal of the canonical form of $A^{\ell}$ are either primitive or $0$.
\end{Lem}

\begin{proof}
(i)  For $1\leq k\leq m$, let $r_k>0$ be the smallest integer such that in $A^{r_k}$, the entry $a_{kk}^{(r_k)}>0$; also let $r$ be the least common multiple of $r_1,\cdots, r_m$. Then $a_{kk}^{(r)}>0$ for each $k$. This implies that if $a_{ij}^{(rn_0)}>0$ for some $n_0$, then $a_{ij}^{(rn)}>0$ for any $n\geq n_0$. For any $i\ne j$, let $r_{ij}$ be such that $a_{ij}^{(rr_{ij})}>0$, then $n=r\prod_{i,j=1, i\ne j}^mr_{ij}$ is the desired integer.

\vspace{0.1cm}

To prove (ii), we use induction on the order $m$ of $A$. It is trivial for $m=1$. Assume it is also true for $m-1$. Consider order $m$. If $A^n$ is primitive for some $n$, then we are done. Otherwise, by (i), there exists $n_0>0$ such that $A^{n_0}$ is  not irreducible. Let $A_1$ be the block matrix on the diagonal of the canonical form of $A^{n_0}$, if it is not zero, then it is irreducible. By induction hypothesis, there exists $n_1$ such that $A_1^{n_1}$ satisfies (ii). Consider the matrix $A_2$ obtained by deleting the rows and columns of $A^{n_0}$ corresponding to $A_1$. Then by using the induction hypothesis again, there exists $n_2$ such that $A_2^{n_2}$ satisfies (ii). By letting $n=n_0n_1n_2$, we conclude that  $A^n$ satisfies (ii) and completes the proof.
\end{proof}

\bigskip

\begin{Prop} \label{th4.2}
Let $(X, \E)$ be a simple $N$-ary augmented tree. Let $T$ be a horizontal component, and let $A$ be the  incidence matrix of the subgraph $T_{\D}$. If $A$ is primitive, then $\partial (T_\D, \E)\simeq \partial (X, \E_v)$.
\end{Prop}

\begin{proof}
Let $\{ [T_1], \cdots , [T_m]\}$ be the equivalence classes in $T_{\D}$, let $u_i = \#T_i$ be the number of vertices in $T_i$,  and let $u = \hbox {gcd}(\bu)$. The proof follows from the same idea as  Theorem 3.7 in \cite{LuLa13} for $(X, \E)$ where $\hbox {gcd}(\bu) =1$ . Here we only sketch the main idea.

\vspace {0.1cm}

 By Proposition \ref{th3.3}, there exists $k$ such that $A^k$ is $(uN^k, \bu)$-rearrangeable.  In view of Proposition \ref{th2.9}, we can assume without loss of generality that $k=1$. Hence for any $T_i$, we have a  $C_i$ to rearrange its descendants into $p_i=u_i/u$ groups consisting of the $T_j$'s, we denote them by ${\mathcal V}_k, 1\leq k \leq p_i$, the number of vertices in ${\mathcal V}_k$ is $uN$.

\vspace {0.1cm}

Let $\ell = \#T$,  let $Y$ be the union of $\ell$ copies of $(X, \E_v)$. Let $\E'$ be an augmented structure on $Y$ by adding  horizontal edges that joining $u$ consecutive vertices in each level  (see the left figure in Figure \ref{map-prop4.2}). (Note that number of vertices in the $n$-th level is  $\ell N^{n-1}$ and $u$ divides $\ell$.)   Then
$$
\partial (Y, \E') \simeq \partial (Y, \E_v) \simeq \pt(X, \E_v)
$$
as the first $\simeq$ follows from a direct check that the identity map is a near-isometry, and the second $\simeq$ follows from Proposition \ref{th2.8}.

\begin{figure}[h]
    \centering
    \includegraphics[width=5.5in]{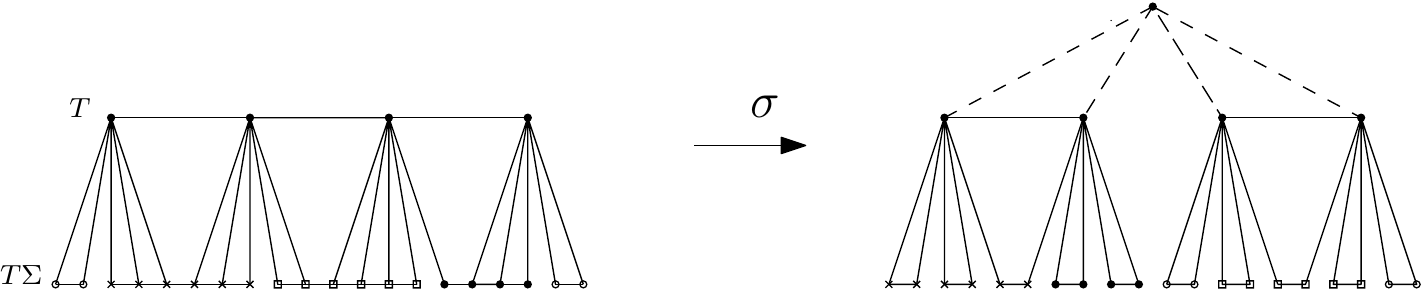}

    \medskip

    \caption{An illustration of $\sigma: T_\D \to Y$ with $u=2, \ell=4$, the $\bullet,\times,\circ,\Box$  denote four kinds  of components.}\label{map-prop4.2}
\end{figure}

\vspace {0.1cm}

With this setup, we can define a map $\sigma : (T_\D, \E) \to (Y, \E')$ as follows. On the first level, let $\sigma$ be any bijection from $T$ to $Y_1$. Suppose we have defined the $T_i$ of $T_\D$ in the $n$-th level, i.e., for $T_i = \{\bi_1, \cdots , \bi_t\}$, and $\sigma(T_i) = (\bj_1 = \sigma (\bi_1), \cdots , \bj_t = \sigma (\bi_t))$, we define $\sigma$ on $T_i \Sigma$ by assigning the vertices of ${\mathcal V}_k$ consecutively to the descendants of  $\sigma(T_i)$ (see Figure \ref{map-prop4.2}). It follows from the rearrangement property that each $\sigma ({\mathcal V}_k)$ are descendants of $u$ consecutive vertices in $\sigma (T_i) (\subset Y_n)$ (see Theorem 3.7 in \cite{LuLa13} for detail).  By the same proof as Theorem 3.7 in \cite{LuLa13}, that $\sigma$ is a near-isometry, and hence $\partial(T_\D, \E) \simeq \partial (Y, \E') \simeq \partial (X, \E_v)$.
\end{proof}

\bigskip

\begin{Rem}\label{remark 4.3}
It follows from the above that  there is a near-isometry $\sigma : (T_\D, \E)\to (Y, \E_v)$ where $(Y, \E_v)$ is the union of $\ell$ copies of $(X, \E_v)$ and $\ell = \#T$. (Actually we can take any finite copies of $(X, \E_v)$ according to Proposition \ref{th2.8}.)
\end{Rem}

\medskip

\begin{Lem}\label{th4.3}
Let $(X, {\mathcal E})$ be a simple $N$-ary augmented tree with equivalence classes $\{[T_1],\dots, [T_m]\}$,  and the incidence matrix is of the form
$$
A= \left[\begin{array}{cc}
A_1 & A_3 \\
0 & A_2
\end{array}\right]
$$
where $A_1, A_2$ are non-zero matrices with orders $r$ and $m-r$ respectively. Let $u_i=\#T_i$, ${\mathbf u}_1=[u_1,\dots,u_r]^t$ and $u= \hbox {gcd}(\bu)$. Suppose

\ (i) $A_1$ is $(uN, {\bu}_1)$-quasi-rearrangeable;

(ii) For $i=r+1,\dots, m$, there exist near-isometries $\sigma_i: ({(T_i)}_D, {\mathcal E}) \to (Y_i, {\mathcal E}_v)$ as in \\
\indent \hspace {0.5cm} Remark \ref{remark 4.3}.

\vspace{0.1cm}

\noindent  Then there exists a near-isometry $\sigma : (X, {\mathcal E}) \to (X, \E_v)$, hence
$$
\partial(X, \E) \simeq \partial(X, \E_v).
$$
\end{Lem}

\begin{proof}  For convenience, we assume that $A_1$ is $(N, {\bu}_1)$-quasi-rearrangeable, the general case follows from the same argument as in last proposition.   We will use (i) and (ii) to construct a near-isometry $\sigma: (X, {\mathcal E}) \to (X, {\mathcal E}_v)$. We write $X_1=(X, {\mathcal E})$ and $X_2=(X, {\mathcal E}_v)$.    Let
$$
\sigma(o)=o \quad \hbox {and} \quad \sigma(i)=i, ~i\in \Sigma.
$$
Suppose  $\sigma$ has been defined on $\Sigma^n$ such that

(1) for  component $T \in[T_i],\ i\leq r$,  $\sigma(T)$ has the same parent, i.e.,
\begin{equation*}
\sigma (x)^{-} = \sigma (y)^{-} \qquad  \forall \ x, y \in T \subset \Sigma^n.
\end{equation*}

(2) for component $T \in[T_i],\ i\geq r+1$, $\sigma(x)=\sigma_i(x)$ for  $x\in T_{\mathcal D}$.

\vspace {0.1cm}
\noindent To define the map $\sigma$ on $\Sigma^{n+1}$, we note that if $T\subset \Sigma^n$ in (2), then $\sigma$ is well-defined by $\sigma_i$. If  $T\subset \Sigma^n$ in (1), without loss of generality, we let $T\in [T_1]$. Then  $T$  gives rise to horizontal components in $\Sigma^{n+1}$, we group them into  ${\mathcal Z}_{1,j}, j =1, \cdots m$ according to  the components belonging to $[T_j]$.

By the quasi-rearrangeable property of $A_1$ (assumption (i)), for the row vector ${\mathbf a}_1=[a_{11},\dots, a_{1r}]$, there exists a nonnegative integral matrix
$C=[c_{sj}]_{u_1\times r}$  such that
$$
{\mathbf a}_1 = {\mathbf 1} C \quad \hbox {and} \quad C{\mathbf  u}_1 \leq [N,\dots,N]^t.
$$
By using this, we can decompose ${\mathbf a}_1$ into $u_1$ groups as follows. Note that $a_{1j}$ denotes the number of horizontal components  that belong to $[T_j]$. For each $1\leq s \leq u_1$, we choose  $c_{sj}, 1\leq j\leq r$, of those components  that are of size $u_j$, and denote by $\Lambda_s$. Then ${\bigcup}_{j=1}^r {\mathcal Z}_{1,j}$ can be rearranged into $u_1$ groups
\begin{equation}\label{eq4.3}
{\cup}_{j=1}^r {\mathcal Z}_{1,j}= \Lambda_1\ \cup \ \cdots \cup \Lambda_{u_1}.
\end{equation}
and the total vertices in each group is  $\leq N$.

\vspace{0.1cm}
For the  component  $T = \{{\mathbf i}_1,  \dots ,  {\mathbf i}_{u_1}\} \subset \Sigma^n$ in $(X, \mathcal E)$, we have defined  $ \sigma (T) = \{{\mathbf j}_1 = \sigma ({\mathbf i}_1), \ \dots ,\ {\mathbf j}_{u_1} = \sigma ({\mathbf i}_{u_1})\}$ in $(X, \mathcal E_v)$ by induction. In view of (\ref {eq4.3}), we define $\sigma$ on ${\bigcup}_{j=1}^r {\mathcal Z}_{1,j}$ by assigning   vertices in $ \Lambda_s$  (cardinality $\leq N$) to the descendants of ${\mathbf j}_s$ (cardinality $N$) in a one-to-one manner;  for the remaining  $T'\in {\bigcup}_{j=r+1}^m {\mathcal Z}_{1,j}$ (maybe empty), say $T'\in [T_j]$ and $j\geq r+1$,  we define for $x \in T'$, $\sigma (x)$ to be any point in $\sigma(T)\Sigma \setminus {\bigcup}_{j=1}^r \sigma({\mathcal Z}_{1,j})$ to fill up the $\sigma(T)\Sigma$ (see Figure \ref{map-lem}).  We also use $\sigma_i$ to induce a near-isometry $\sigma : T_\D \to (\sigma(T))_\D$.  We apply the same construction of $\sigma$ on the offsprings of every component in $\Sigma^{n+1}$. Inductively, $\sigma$ can be defined from $X_1$ to $X_2$.

\begin{figure}[h]
    \centering
    \includegraphics[width=5.0in]{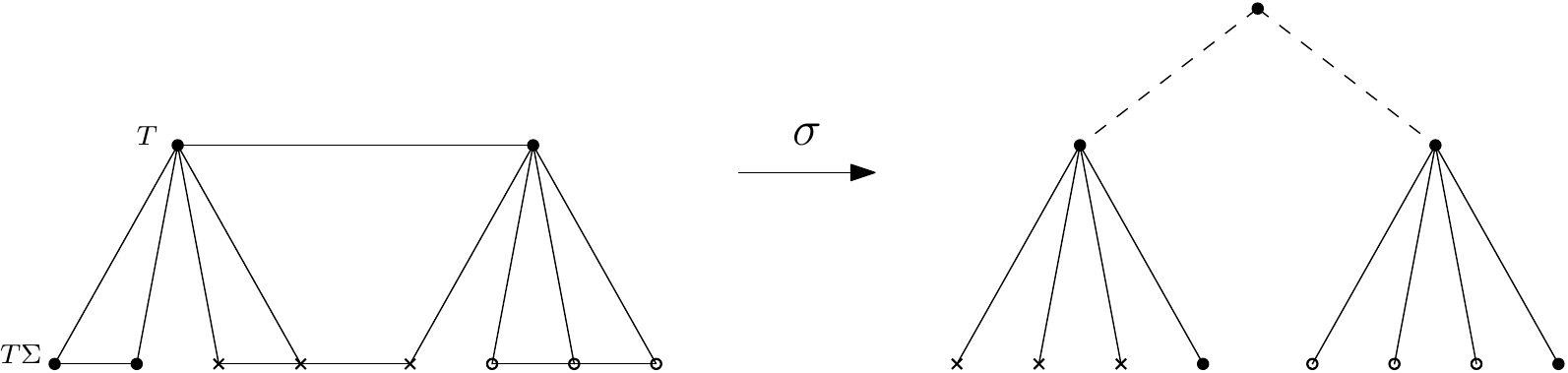}

    \medskip

    \caption{An illustration of a rearrangement by $\sigma$, the $\bullet,\times,\circ$ denote the three kinds of  components. The first component in $\bullet$ is $T'$ which belongs to the second type, and the other two components belong to the first type }\label{map-lem}
\end{figure}

Finally we show that  $\sigma: X_1 \to X_2$ is indeed a near-isometry. Let $\pi({\mathbf {x,y}})$  be the canonical geodesic in $X_1$,
it can be written as
$$
\pi({\mathbf {x,y}})=[{\mathbf x}, {\mathbf v}_1,\dots, {\mathbf v}_n, {\mathbf t}_1,\dots, {\mathbf t}_k, {\mathbf w}_n,\dots, {\mathbf w}_1, {\mathbf y}]
$$
where $[{\mathbf t}_1,\dots, {\mathbf t}_k]$ is the horizontal part and $[{\mathbf x}, {\mathbf v}_1,\dots, {\mathbf v}_n, {\mathbf t}_1],~ [{\mathbf t}_k, {\mathbf w}_n,\dots, {\mathbf w}_1, {\mathbf y}]$ are vertical parts. Clearly, $\{{\mathbf t}_1,\dots, {\mathbf t}_k\}$ must be included in one horizontal component of $X_1$, we denote it by $T$ and let $T\in [T_i]$ for some $1\le i\le m$.

If $i\geq r+1$,  it is clear that $$\big |d(\sigma({\mathbf x}),\sigma({\mathbf y}))-d({\mathbf {x,y}})\big |=\big |d(\sigma_i({\mathbf x}),\sigma_i({\mathbf y}))-d({\mathbf {x,y}})\big |\leq c,$$
where $c$ is the uniform bound of the near-isometries $\{\sigma_i\}_{i={r+1}}^m$.

If $i\le r$,  consider the position of ${\mathbf x}$ first: if ${\mathbf x}\in T'\in [T_j], j\le r$, then $|\sigma({\mathbf x})|=|{\mathbf x}|$ by the construction of $\sigma$, hence $\big| |\sigma({\mathbf x})|-|{\mathbf x}| \big|=0$;  otherwise ${\mathbf x}\in T'_{\D}$ for some $T'\in [T_j]$ and $j\ge r+1$, in this case, we have $\big| |\sigma({\mathbf x})|-|{\mathbf x}| \big|\le c$ as above. Similarly for ${\mathbf y}$. Notice that $d({\mathbf x}, {\mathbf y})=|{\mathbf x}|+|{\mathbf y}|-2\ell+h$, where $\ell$ and $h$ are the level and the length of the horizontal part of the canonical geodesic $\pi(x,y)$ (see \cite{LaWa09} or \cite{LuLa13}). Therefore
\begin{eqnarray*}
\big |d(\sigma({\mathbf x}),\sigma({\mathbf y})) - d({\mathbf {x,y}})\big | &\leq & \big||\sigma({\mathbf x})|-|{\mathbf x}|\big|+\big||\sigma({\mathbf y})|-|{\mathbf y}|\big|+2|l'-l|+|h'-h| \\
&\leq & 2c+2|l'-l| +|h'-h|\\
&\leq & 2c + 2|l'-l|+c_0
\end{eqnarray*}
where $c_0=\max_{1\le i\le r} u_i$. Moreover, by condition (1), it follows that  $|l'-l|\le \frac{1}{2}(c_0 +1)$ (see also \cite{LuLa13}).  Consequently
$$
\big |d(\sigma({\mathbf x}),\sigma({\mathbf y})) - d({\mathbf x},{\mathbf y}) \big |\leq 2(c+c_0)+1.
$$
This completes the proof that $\sigma$ is a near-isometry.
\end{proof}

\bigskip

\begin{theorem} \label{th4.4}
Suppose $(X, {\mathcal E}_v)$ is an $N$-ary tree, and the augmented tree $(X, {\mathcal E})$ is simple. Then $\partial(X, {\mathcal E})\simeq \partial(X, {\mathcal E}_v).$
\end{theorem}

\begin{proof}
Let  $\{[T_1],\dots, [T_m]\}$ be the equivalence classes of horizontal components,  $u_i=\#T_i$, and $A$ the associated incidence matrix. By Lemma \ref{th4.1}, there exists $\ell\geq 1$ and a permutation matrix $P$ such that
$$
A^\ell=\left[\begin{array}{cccccc}
A_1 &  & {\bf *}  \\
 &  \ddots & \\
  {\mathbf 0} &  &  A_k
\end{array}\right]
$$
where $A_i$ are either $0$ or primitive. From the definition of incidence matrix, we see that $A_k \neq 0$, hence is primitive. Without loss of generality, we let $\ell=1$ (by Proposition \ref{th2.9}).

If $k=1$, then $A=A_1$ is primitive. For any horizontal component $T \subset \Sigma$, $T_\D$ has incidence matrix $A$ also. Hence by Proposition \ref{th4.2} that $\partial (T_\D, \E)  \simeq \partial (X, \E_v)$. As $\Sigma$ is the disjoint union of such $T$, it follows from Proposition \ref {th2.8} and Remark \ref{remark 4.3} that  $\partial (X, \E) = \partial (\cup (T_\D, \E)) \simeq \partial (X, \E_v)$.

\vspace{0.1cm}

If $k=2$, let $A_1$, $A_2$ correspond to  $\{[T_1],\dots, [T_r]\}$,   and  $\{[T_{r+1}],\dots, [T_r]\}$ respectively. If $A_1={\mathbf 0}$, we can take $A_2$ as the incidence matrix of $(X, {\mathcal E})$ by removing finitely many vertices that belong to $[T_i], 1\leq i\leq r$. By Proposition \ref{th4.2}, the result follows. If $A_1 \neq {\bf 0}$, then  Proposition \ref {th4.2} and Remark \ref{remark 4.3} imply that  assumption (ii) in Lemma \ref{th4.3} is satisfied; the other assumptions also follow readily, and the theorem follows.

The general case that $k\geq 2$ follows by applying the above argument inductively.
\end{proof}

\bigskip

\section{\bf Applications to self-similar sets}

In this section, we will make use of the previous results to study the Lipschitz equivalence of self-similar sets and their unions. As before we assume the self-similar set $K$ is generated by an IFS  $\{S_i\}_{j=1}^N$ on $\R^d$ where
\begin{equation}\label{eq5.1}
S_i(x)=rR_ix+d_i,\quad  1\le i\le N
\end{equation}  with $0<r<1$, $R_i$ orthogonal matrices and $d_i\in \R^d$.  The representing symbolic space of $K$ is the tree $X= \cup_{n=0}^\infty \Sigma^n$ where $\Sigma=\{1,\dots, N\}$,  and $\Sigma^n$ is the set of indices $\bi = i_1i_2\cdots i_n$, representing $S_\bi=S_{i_1}\circ\dots\circ S_{i_n}$. Let $(X, {\mathcal E}_v)$ be as before, we define the horizontal edge set ${\mathcal E}_h$ to be
\begin{equation}\label{eq2.4'}
\E_h = \{ (\bi, \bj) :\  |\bi| = |\bj|,  \ \bi \not = \bj  \ \ \hbox {and} \ \  S_\bi(K) \cap S_\bj(K) \not = \emptyset\},
\end{equation}
and let $\E = \E_v\cup \E_h$.  Then $\E_h$ satisfies (\ref{Def}); we call $(X, \E)$ an {\it ($N$-)self-similar augmented tree}. A sufficient condition for $(X, \E)$ to be  hyperbolic is that the IFS satisfies the {\it open set condition (OSC)} \cite{LaWa09} (see \cite{Wa12} for the more general situations). In the special case that the IFS is strongly separated, i.e., $S_i(K) \cap S_j(K) = \emptyset$ for $i\ne j$, then $\E = \E_v$, and $\pt X$ (also $K$) is a homogeneous Cantor-type set (it is also called a dust-like self-similar set in  \cite{FaMa92, LuLa13}).

%

\medskip

Let $J$ be a nonempty  bounded closed invariant set, i.e., $S_i(J)\subset J$ for all $i$.  For indices ${\bi}\in X$, denote  $J_{\bi}=S_{\bi}(J)$. The self-similar set $K$ (or the IFS) is said to satisfy {\it condition (H)} if there exists a bounded closed invariant set $J$  such that for any integer $n\ge 1$ and indices $\bi, \bj \in \Sigma^n$, then
$$J_{\bi} \cap J_{\bj}=\emptyset \ \Rightarrow  \  \dist(J_{\bi}, J_{\bj})\ge cr^n \quad \text{for some} \ c>0.$$
In many situations, we can take $J=K$, or take $J=\overline{U}$ for the open set $U$ in the OSC (see \cite{LuLa13, Lu13}). It was proved that if the augmented tree $(X,\E)$ is simple and $K$ satisfies condition (H), then the natural map $\Phi: \partial X \to K$ satisfies the second inequality of the following the H\"older equivalent property:
$$
C^{-1}|\Phi(\xi)-\Phi(\eta)|\le \rho_a(\xi,\eta)^\alpha \le C|\Phi(\xi)-\Phi(\eta)|, \quad \forall \  \xi\ne \eta \in \partial X,
$$
where $\alpha=-\log r/a$ and $C>0$ is a constant (\cite {LuLa13}, Proposition 3.5). (The first inequality always holds.)

\medskip

Condition (H) is satisfied by the standard self-similar sets, for example, the generating IFS satisfies the strongly separation condition, or the OSC and all the parameters of the similitudes are rational numbers (for more discussions on this condition, we refer to \cite{LaWa09, Wa12}).  Note that all the IFSs considered here satisfy condition (H).

\medskip

\begin{theorem}\label{th5.1}
Let $K,K'$ be two $N$-self-similar sets that are generated by two IFSs with the same contraction ratio $r$ (as in (\ref{eq5.1})) and satisfy condition (H). If their associated augmented trees are simple, then $K\simeq K'$.

In particular, $K$ and $K'$ are  Lipschitz equivalent to the $N$-Cantor set with contraction ratio $r$.
\end{theorem}

\begin{proof} We remark that the theorem is a modification of Theorem 3.10 in \cite{LuLa13} by omitting the rearrangeable condition on $A$, as it is not necessary in view of Theorem \ref{th4.4}.  The proof is the same,  we just sketch the main idea for completeness. Let $(X,\E), \ (Y,\E)$ be the two augmented trees induced by $K, K'$ respectively. Then Theorem \ref{th4.4} implies that
$$
\pt(X,\E)\simeq \pt(X,\E_v)=\pt(Y,\E_v)\simeq \pt(Y,\E);
$$
we let $\varphi$ denote the bi-Lipschitz mapping from $\pt(X,\E)$ onto $\pt(Y,\E)$. Moreover, there exist H\"older equivalent maps  $\Phi_1:\,\pt(X,\E)\to K$ and $\Phi_2:\,\pt(Y,\E)\to K'$ (depend on the parameter $a$ in the hyperbolic metric $\rho_a$). Then $\tau=\Phi_2\circ\varphi\circ\Phi_1^{-1}$ is the desired bi-Lipschitz map from $K$ onto $K'$.

The last statement also follows if we treat $(X,\E_v)$ as an augmented tree induced by a strongly separated IFS.
\end{proof}

\medskip

\begin{Cor} The IFSs in Theorem \ref{th5.1} satisfy the open set condition.
\end{Cor}

\begin{proof} It follows from the same proof as Corollary 3.11 in \cite {LuLa13}. The Lipschitz equivalence implies that $0< {\mathcal H}^s(K) <\infty$ where $s$ is the dimension of $K$, which implies the open set condition is satisfied by Schief's criterion \cite{Sc94}.
\end{proof}

\medskip
\begin{Cor}\label{th5.2}
Let $\{K_i\}_{i=1}^\ell$ be a sequence of self-similar sets satisfying  the assumptions of Theorem \ref{th5.1}.  Let $X=\{o\}\cup (\cup_{n\ge0}\{1,2,\dots,\ell\}\Sigma^n)$ be a tree, and be equipped with a horizontal edge set
$$
\E_h=\{(i\bi,j\bj):\,|\bi|=|\bj|,(K_i)_\bi\cap (K_j)_\bj\ne\emptyset, \ 1\le i,j\le k,\ \bi,\bj\in \cup_{n\ge0}\Sigma^n\}.
$$
Suppose the union $\cup_{i=1}^\ell K_i$ satisfies condition (H) in the sense that: if $(K_i)_\bi\cap (K_j)_\bj=\emptyset$ for any $\bi, \bj\in \Sigma^n$ and $n\ge 1$ then $\dist((K_i)_\bi, (K_j)_\bj)\ge cr^n$ for some $c>0$.  Then $\cup_{i=1}^\ell K_i\simeq K_j$ for all $j$, provided the augmented tree $(X,\E_v\cup \E_h)$ is simple.
\end{Cor}

\begin{proof}
Since $\pt X \simeq \pt X_j$ and  $\pt X_j$ is H\"older equivalent to $K_j$,  it suffices to show that $\pt X$ and $\cup_{i=1}^\ell K_i$ are also H\"older equivalent. We omit the proof as it  is the same as that of Proposition 3.5 in \cite{LuLa13} with some minor modifications.
\end{proof}

\bigskip

In the following, we will illustrate our results by some simple examples.  The first one was raised in \cite {LuLa13} (see also \cite{XiXi13}) that its incidence matrix is not primitive.

\bigskip

\begin{Exa}\label{exa5.3}
Let $K$ be a self-similar set generated by $\{S_1(x)=\frac{x}{5},S_2(x)=-\frac15(x-4), S_3(x)=\frac15(x+4)\}$ (see Figure \ref{145}).  Then $K$ is Lipschitz equivalent to a $3$-Cantor set.
\end{Exa}

It is easy to see that the equivalence classes are $[1], [2,3]$. Hence  the augmented tree for $K$ is simple, and the incidence matrix is
        $$A=\begin{bmatrix}
          1&1\\
          0&3
        \end{bmatrix}.$$
 By Theorem \ref{th5.1}, $K$ is Lipschitz equivalent to the $3$-Cantor set $K'$ generated by  $\{S'_i(x)=\frac15(x+2(i-1))\}_{i=1}^3$.

\begin{figure}[h]
    \centering
    \subfigure[]{
    \includegraphics[width=4cm]{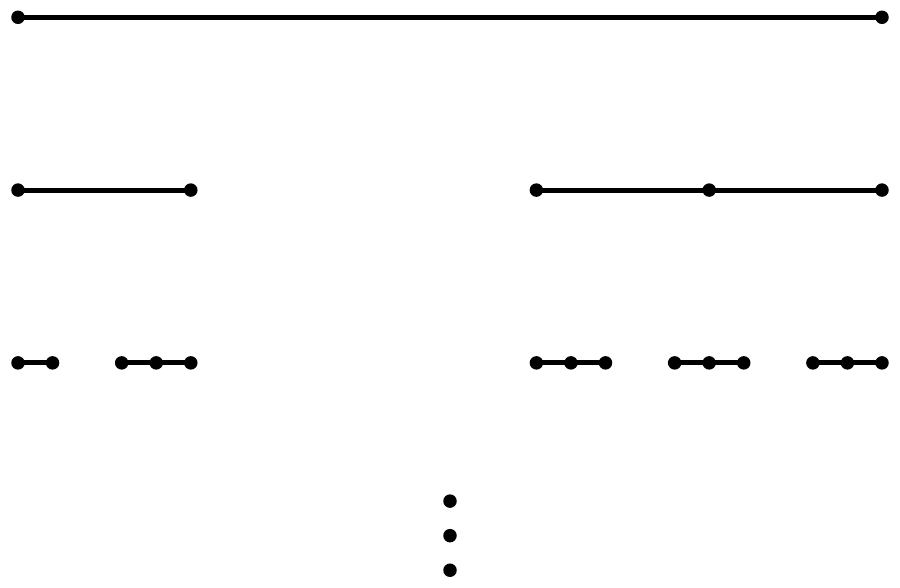}
    }\qquad\qquad
    \subfigure[]{
     \includegraphics[width=4cm]{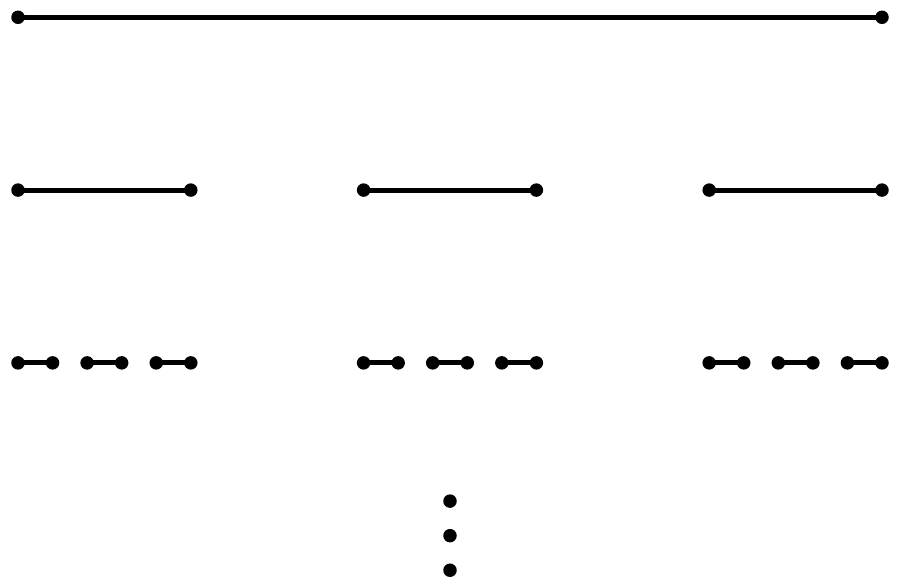}
     }
\caption{Self-similar sets $K$ and $K'$ of Example \ref{exa5.3}.}\label{145}
\end{figure}

\bigskip

The next two examples concerns the union of Cantor-type sets, which are not self-similar.

\begin{Exa}\label{exa5.4}
Let $0<r<1/2$,
$$
d_1= {\tiny \begin{bmatrix}  0\\ 0  \end{bmatrix}} , \quad
d_2={\tiny \begin {bmatrix} r^{-1}-1 \\ r^{-1}-1 \end{bmatrix}},  \quad d_3={\tiny \begin {bmatrix} c_1 \\ r^{-1}-1-c_2 \end{bmatrix}},  \quad d_3'={\tiny \begin {bmatrix} r^{-1}-1-c_3 \\ c_4 \end{bmatrix}}
$$
where $0< c_1,c_2,c_3,c_4 <r$. Let  $\D_1=\{d_1,d_2,d_3\}$, $\D_2=\{d_1,d_2,d_3'\}$ and let $K$, $K'$ be the self-similar sets generated by $\{S_j(x)=r(x+d_j):\,d_j\in \D_1\}$ and  $\{S_j(x)=r(x+d_j):\,d_j\in \D_2\}$, respectively. Then $K\cup K'\simeq K\simeq K'$.(See Figure \ref{union}.)
\end{Exa}

\begin{proof}
 Note that $K$ and $K'$ are Lipschitz equivalent to the $3$-Cantor set, and they coincide in the diagonal, more precisely, $K_\bi\cap K'_\bj\ne\emptyset$ if and only if $\bi=\bj\in\cup_{n\ge0}\{1,2\}^n:=\Sigma_2^*$.  Let $X=\{o\}\cup\big(\cup_{n\ge0}\{1,2\}\Sigma^n\big)$ with $\Sigma = \{1,2,3\}$,  and define the set of horizontal edges by
 $$
 \E_h=\{(1\bi,2\bi):  \bi\in\Sigma_2^*\}.
 $$
Then $(X,\E)$ where $\E=\E_v\cup \E_h$ is the augmented tree induced by $K\cup K'$. The equivalence classes of $(X,\E)$ are  $[1,2],[3]$, and the corresponding incidence matrix is
$$
A=\begin{bmatrix}
          2&2\\
          0&3
        \end{bmatrix}.
$$
Hence the augmented tree is simple and the conclusion follows by Corollary \ref{th5.2}.
\end{proof}

\begin{figure} [h]
    \centering
     \subfigure[]{
    \includegraphics[width=1.5in]{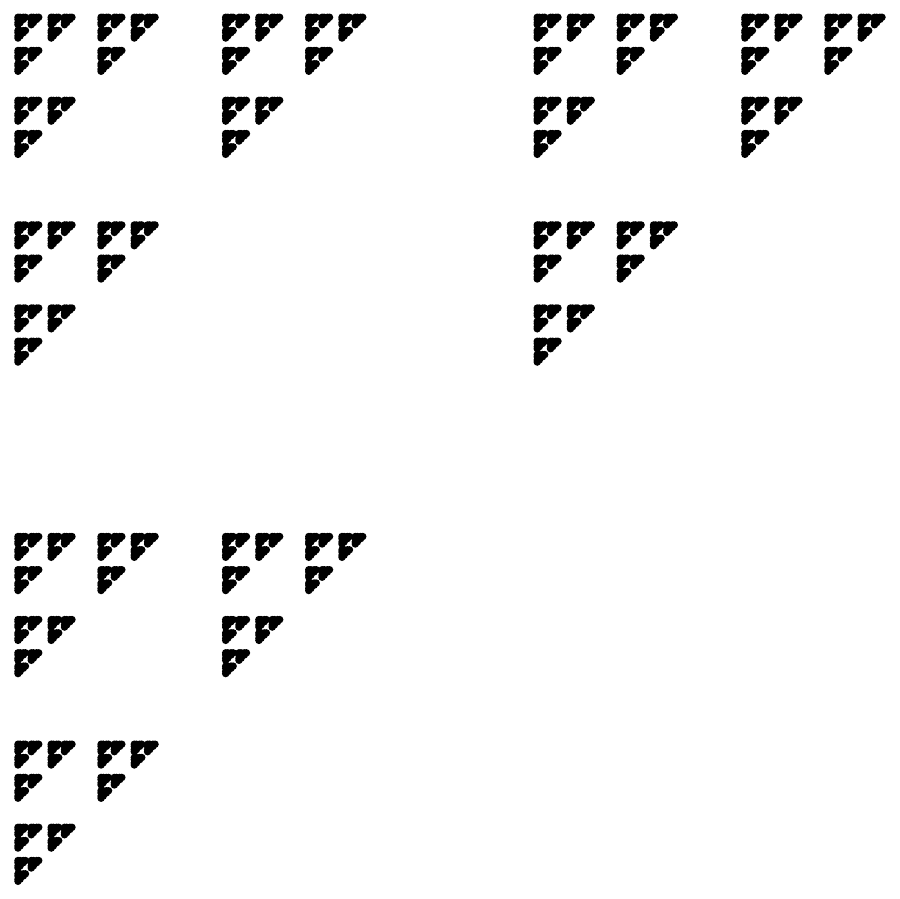}
    }
     \subfigure[]{
      \includegraphics[width=1.5in]{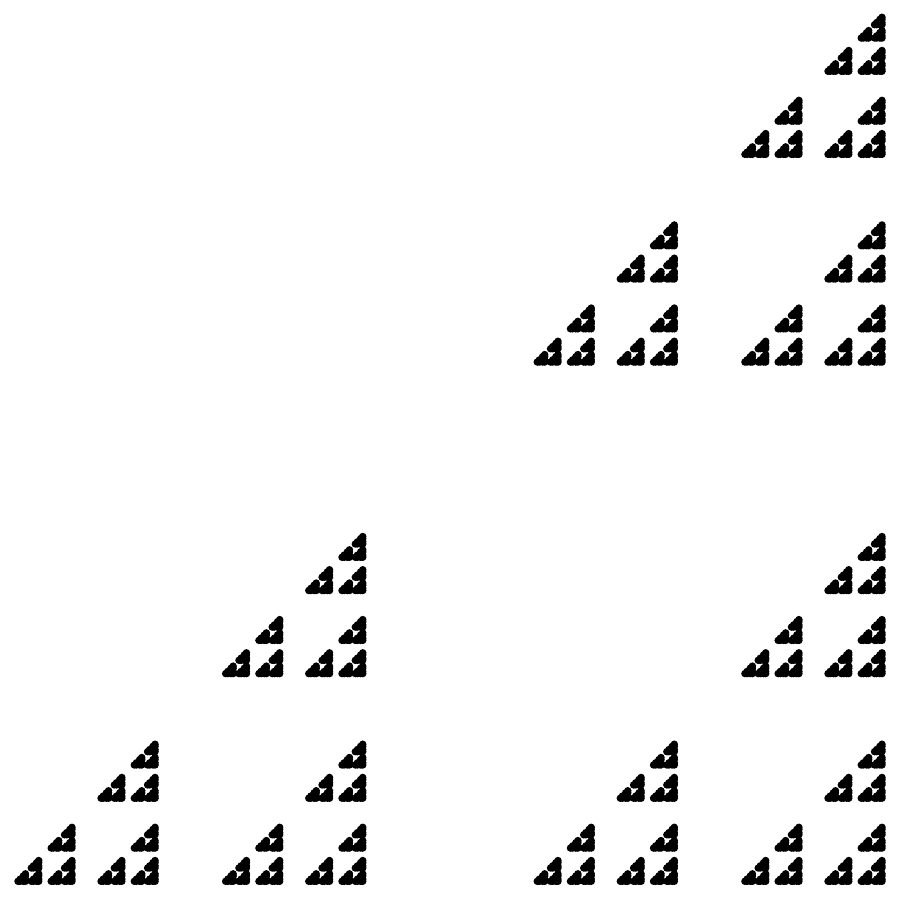}
     }\qquad \qquad
      \subfigure[]{
       \includegraphics[width=1.5in]{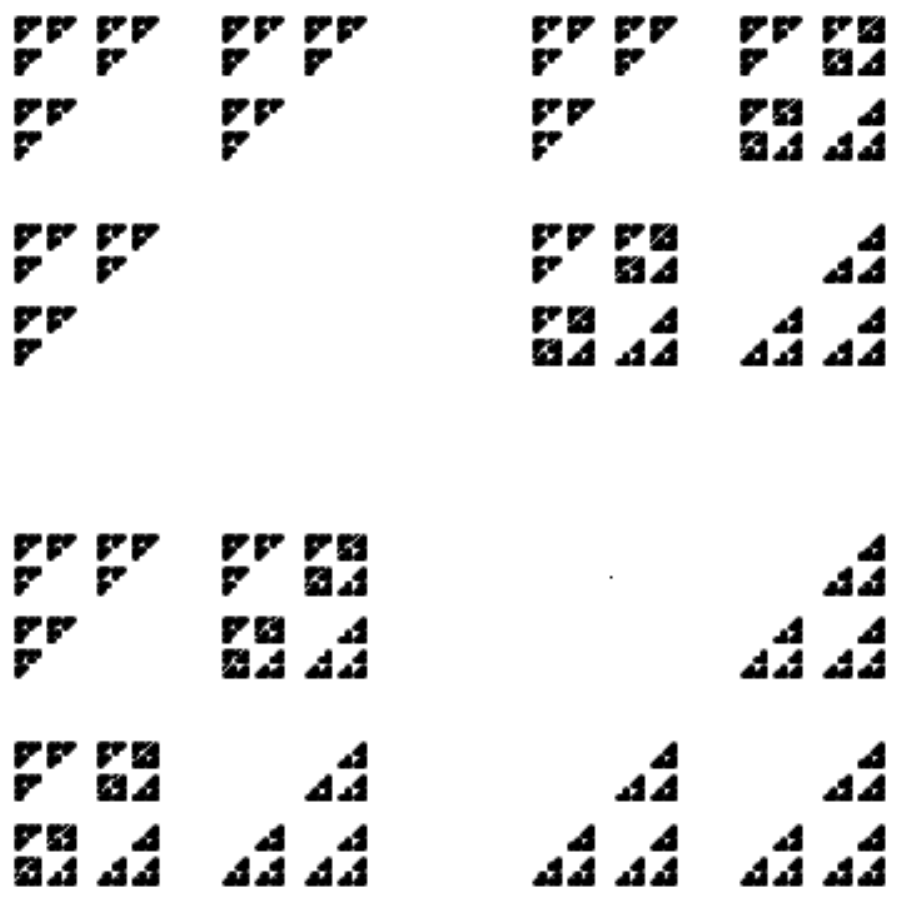}
      }
\caption{Example \ref{exa5.4}: $K$, $K'$ and $K\cup K'$ for $r=2/5,c_i=0, i=1,\dots,4$.}\label{union}
\end{figure}

\begin{Prop}\label{exa5.5}
Let $\mathcal C$ be the standard Cantor set for which the IFS is
$$
S_0(x)=x/3, \quad S_2(x)=(x+2)/3.
$$
Then $\mathcal C\cup(\mathcal C+\alpha)\simeq \mathcal C$  when  $\alpha>1$ or $0\le\alpha\le 1$ is a rational.
\end{Prop}

\begin{proof}  Let $K_a = \mathcal C$ and $K_b = \mathcal C +\alpha$. When $\alpha>1$, the two sets are disjoint, the conclusion follows  from Proposition \ref{th2.8} for the disjoint union of two trees. For $\alpha =1$, then $K_a$ touches  $K_b$  at one point. We can use the same argument as in Example \ref{exa5.4} to show that the augmented tree induced by $K_a \cup K_b$ is simple, and has incidence matrix
$$
A=\begin{bmatrix}
1&2\\0&2
\end{bmatrix},
$$
which yields $K_a\cup K_b\simeq \mathcal C$.

Now we consider the rational $\alpha \in (0,1)$. For the radix expansion $\alpha = \sum_{n=1}^\infty \frac{\alpha_n}{3^n}$, with $\alpha_n \in \{ -2, 0,2\}$, it is easy to check that $\alpha$ has two expansions if and only if $\alpha = \sum_{n=1}^k \alpha_n3^{-n}\pm 3^{-k}$, and the representing sequences are
\begin{equation}\label{eq5.2}
\alpha_1\cdots \alpha_{k-1}0 22\cdots \quad \hbox {and} \quad  \alpha_1\cdots \alpha_{k-1}2{(-2)}(-2) \cdots.
\end{equation}

Let $\D=\{0,2\}$ and $\D_n=\D+3\D+\cdots + 3^{n-1}\D, n\ge 1$. Then the Cantor set $\mathcal C$ satisfies the following set equation for all $n\ge 1$: $${\mathcal C}=\bigcup_{d\in \D_n}\frac 1{3^n}({\mathcal C}+d).$$
Hence $3^{k+1} (K_a\cup K_b)$ can be written as $$3^{k+1} (K_a\cup K_b)=3^{k+1}\big({\mathcal C}\cup({\mathcal C}+\alpha)\big)=\bigcup_{d,d'\in \D_{k+1}}{\mathcal C}\cup ({\mathcal C}+d-d'+3^{k+1}\alpha)$$
which is a finite disjoint union of translates of $\mathcal C$ and ${\mathcal C}\cup({\mathcal C}+1)$.  By the last part and Corollary \ref{th5.2}, we conclude that  $K_a\cup K_b\simeq \mathcal C$.

It remains to consider the case that the rational $\alpha\in(0,1)$ has a unique expression. The representing sequence has the form
\begin{equation} \label {eq5.3}
\alpha_1\alpha_2\cdots=\alpha_1\cdots\alpha_N\eta \eta \cdots,
\end{equation}
where $\eta = \alpha_{N+1} \cdots \alpha_{N+M}$, and is not equal to the expression in \eqref{eq5.2}. For convenience, we denote the symbolic space of $\mathcal C$ by $\Sigma^* =\cup_{n\ge0}\Sigma^n$ with $\Sigma=\{0,2\}$, and let $ X= \{o\} \cup  \{a,b\}\Sigma^*$ be the union tree for $K_a\cup K_b$. Let $I=[0,1]$ be the unit interval and $I_\bi=S_\bi(I)$, then for any $\bi,\bj\in \Sigma^n$, $I_\bi\cap I_\bj \neq \emptyset$ if and only if $I_\bi=I_\bj$. Define
$$
\E_h=\{(a\bi,b\bj):\, |\bi|=|\bj|,\ I_\bi\cap(I_\bj+\alpha)\ne\emptyset\},
$$
and $\E=\E_v\cup\E_h$, then $(X,\E)$ is the augmented tree induced by $K_a\cup K_b$. Note that for distinct $\bi,\bj \in \Sigma^n$, $(a\bi,b\bj)\in\E_h$ is equivalent to $|S_\bi(0)-(S_\bj(0)+\alpha)|<3^{-n}$ (strict inequality  holds, otherwise it will reduce to the previous case that $\alpha$ has two radix expansions)  i.e.,
\begin{equation} \label{eq5.4}
\left|\sum_{k=1}^n\frac{i_k-j_k-\alpha_k}{3^k} -\sum_{k=n+1}^\infty\frac{\alpha_k}{3^n}\right|< \frac1{3^n},
\end{equation}
 which is also equivalent to
\begin{equation} \label{eq5.5}
i_k-j_k=\alpha_k, \quad  1\leq k \leq n
\end{equation}
 in view of the unique radix expansion of $\alpha$.  That (\ref{eq5.4}) together with the diameter of $I_\bi$ or $I_\bj+\alpha$ being $3^{-n}$, implies that any horizontal  component $T$ of $(X,\E)$ must satisfy $\#T\le 2$. Note that all the subtrees  $T_\D$ with  $\#T=1$ are graph isomorphic.  Also we claim that
 \vspace {0.1cm}

 {\it if $T=\{a\bi, b\bj\}, T'=\{a\bi', b\bj'\}$  are two horizontal components  with $\bi, \bj\in \Sigma^n$ and $\bi', \bj\in \Sigma^{n+M}, n>N$,  then $T \sim T'$.}

\vspace {0.1cm}

\noindent  Then all horizontal components are equivalent to those in the first $N+M+1$ levels of $X$. Hence $(X, \E)$ is simple, which yields the desired result by Corollary \ref{th5.2}.

To prove the claim, we first define a map $\sigma : T_\D \to T'_\D$ by
$$
\sigma (a\bi \bu) = a\bi'\bu, \ \ \sigma (b\bj \bv) = b\bj'\bv
$$
and then show that $\sigma$ is a graph isomorphism.  Let $(x,y)$ be a horizontal edge in $T_\D$. Suppose $x=a\bi\bu, y=b\bj\bv$, where $\bi=i_1\cdots i_n, \bu=i_{n+1}\cdots i_{n+k}$ and $\bj=j_1\cdots j_n, \bv=j_{n+1}\cdots j_{n+k}$. Then by \eqref{eq5.5},
we have
$$
i_m - j_m=\alpha_m, \quad 1\le m\le n+k.
$$
Since $(a\bi',b\bj')\in \E_h$, \eqref{eq5.5} again implies $i_m'-j_m'=\alpha_m,\ 1\le m\le n+M$;  also by the definition of $\sigma$ and (\ref{eq5.3}),
$$
i_{m+M}'-j_{m+M}'= i_m- j_m = \alpha_m = \alpha_{m+M} , \quad n+1 \leq m \leq n+k.
$$
By (\ref{eq5.5}), we have  $(a\bi'\bu,b\bj'\bv)\in \E_h$. Therefore $\sigma$ preserves the horizontal edges. Analogously, $\sigma^{-1}$ has the same property. This completes the proof of the claim.
\end{proof}

\bigskip

To conclude, we mention that we only consider the augmented trees arise from the IFS with equal contraction ratio. There are interesting investigations of Lipschitz equivalence of totally disconnected self-similar sets with non-equal contraction ratios  (\cite{RuWaXi12, RaRuWa10, XiRu07, XiXi13}). In particular, in \cite{XiXi13}, Xi and Xiong gave an extensive study of the exponentially commensurable contraction ratios and the open set condition; they obtained a complete classification of such case in terms of the ideal classes. It will be interesting to show how this hyperbolic graph approach can be extended to such cases.

\end{document}